\documentclass[final]{siamltex}
\usepackage{graphicx}
\usepackage{amsmath,amssymb}
\usepackage{subfig}
\usepackage{url}
\usepackage{xcolor}
\definecolor{lightgray}{gray}{0.5}
\definecolor{darkred}{rgb}{.6,0,0}
\definecolor{darkblue}{rgb}{0,0,0.5}
\usepackage[colorlinks=true,citecolor=darkblue,urlcolor=darkblue,linkcolor=darkred]{hyperref}

\renewcommand{\grad}{\nabla}

\newcommand{\lap}{\delsq}
\newcommand{\dx}{\Delta x}
\newcommand{\dt}{\Delta t}
\newcommand{\Real}{\mathbb{R}}

\newcommand{\hf}{\frac{1}{2}}

\newcommand{\uvec}{\mathbf{u}}
\newcommand{\vvec}{\mathbf{v}}

\newcommand{\Emat}{\mathbf{E}}
\newcommand{\Lmat}{\mathbf{L}}
\newcommand{\Imat}{\mathbf{I}}
\newcommand{\Mmat}{\mathbf{M}}
\newcommand{\Amat}{\mathbf{A}}
\newcommand{\Dmat}{\mathbf{D}}

\newcommand{\surf}{\mathcal{S}}
\newcommand{\gradsurf}{\grad_{\!\surf}}
\renewcommand{\lap}{\Delta}
\newcommand{\lapsurf}{\lap_{\surf}}

\newcommand{\cp}{\textup{cp}}

\newcommand{\diverge}{\textup{div}}
\newcommand{\divsurf}{\diverge_{\!\surf}}

\newtheorem{prob}[theorem]{Problem}
\newtheorem{defn}[theorem]{Definition}
\newtheorem{principle}[theorem]{Principle}

\newcommand{\param}{\gamma}

\title{An embedded method-of-lines approach to solving partial
  differential equations on surfaces}

\author{Ingrid von Glehn, Thomas M\"{a}rz, \and Colin B. Macdonald
  \thanks{Mathematical Institute, University of Oxford, OX1 3LB, UK
    (\url{{vonglehn, maerz,macdonald}@maths.ox.ac.uk}).
    This publication is based on work supported by Award No.~KUK-C1-013-04,
    made by King Abdullah University of Science and Technology (KAUST)}
}

\begin{document}

\maketitle

\begin{abstract}
We introduce a method-of-lines formulation of the closest point
method, a numerical technique for solving partial differential
equations (PDEs) defined on surfaces. This is an embedding method,
which uses an implicit representation of the surface in a band
containing the surface. We define a modified equation in the band, obtained in a
straightforward way from the original evolution PDE, and show that the solutions
of this equation are consistent with those of the surface
equation. The resulting system can then be solved with standard
implicit or explicit time-stepping schemes, and the solutions in the
band can be restricted to the surface. Our derivation generalizes
existing formulations of the closest point method and is amenable to
standard convergence analysis.
\end{abstract}

\begin{keywords} 
Closest Point Method, implicit surfaces, partial differential
equations, method of lines
\end{keywords}

\begin{AMS}
65M06, 58J35, 65M20
\end{AMS}

\pagestyle{myheadings}
\thispagestyle{plain}
\markboth{I. VON GLEHN, T. M\"{A}RZ, AND C. B. MACDONALD}{AN EMBEDDED
    METHOD-OF-LINES APPROACH TO SURFACE PDES}

\section{Introduction}
Partial differential equations (PDEs) defined on curved
surfaces appear in a variety of physical and biological
systems and applications. Examples include fluid flow on surfaces
\cite{MCC2012}, the diffusion of chemicals on cell membranes
\cite{Murray2003}, and texture mapping in computer graphics \cite{Turk1991}.

The numerical solution of these equations and treatment of surface
differential operators is an area of active research. Some methods
work directly on the surface, using either a parameterization (for a
survey, see \cite{FH2005}), or a triangulation of the surface
\cite{DE2007a}. Embedding methods form an alternative approach, in
which the surface is embedded into a larger space, and a related
equation is solved in this surrounding space. Finally, a restriction is used to
obtain the solution on the surface. The closest point method
\cite{RM2008,MR2009} is an example of such a technique. Other
embedding techniques using an implicit representation of the surface
include the level set approach of \cite{BCO+2001, Greer2006} for
variational problems, and finite element methods on implicit surfaces
\cite{Burger2009, DDE+2010}. Recently, methods using radial basis
functions \cite{FW2012, Piret2012} have been introduced.

This paper is based on the closest point method, which is applicable
to a wide variety of surface geometries, and is simple to implement
using standard well-studied numerical techniques on Cartesian grids
\cite{RM2008}. It has been applied to a variety of problems, including
eigenvalue problems \cite{MBR2011}, image segmentation \cite{TMR2009}, image
denoising \cite{BGM+2013}, and fluid effects on surfaces
\cite{AMT+2012}.

We derive a modified formulation of the closest point
embedding equation for evolution PDEs, and show that there is a
one-to-one correspondence between solutions of the embedding and surface
equations. This formulation is simple to discretize and solve
numerically using a standard method of lines approach. This
generalizes a stabilized implicit method of \cite{MR2009}, and has the
advantage that it can be adapted to a very general class of
problems. An appropriate explicit or implicit time-stepping scheme can be used,
depending on the particular problem considered.
The new method retains the advantages of the original
closest point method.

\subsection{Outline}
We begin Section~\ref{sec:cpm} with a review of the closest point
method followed by an overview of the new formulation in
Section~\ref{sec:embedding} and an example in
Section~\ref{subsec:diffexample}. Section \ref{sec:systemofeqs} then
defines a system of embedding equations, and shows that this is
consistent with the original surface PDE. This system is reduced to a
single equation in Section \ref{sec:singleeqn}. The numerical solution
of this equation is studied, and a discretization to obtain a system
of ODEs is presented in Section \ref{sec:MOL}. Numerical studies of an
introduced parameter, and convergence studies and examples are given in Sections
\ref{sec:lambda} and \ref{sec:numerics}. Finally we present some
conclusions and a discussion of future work.

\section{The Closest Point Method}
\label{sec:cpm}
Suppose we want to solve a evolutionary PDE defined on a curved surface.
A simple embedding technique known as the closest point method was
introduced in \cite{RM2008}. This method uses the fact that the
surface is embedded in $\Real^n$, and represents the surface by a
retraction based on Euclidean distance. For every point $x$ in this
surrounding space $\Real^n$ the retraction returns a surface point
which is closest to $x$. We call this retraction a closest point
function denoted by $\cp$. If a function is defined on the surface,
the data can be extended \emph{off the surface into the surrounding
  space} by assigning to each point $x$ the value of the surface
function at $\cp(x)$. The key observation is that this function is now
constant in the direction normal to the surface. Surface gradients and
surface divergences of the original function will agree with the
standard Cartesian operators of the extended function at the surface
\cite{MM2012}. We call these ideas the ``gradient principle'' and
``divergence principle''~\cite{RM2008}.

These principles are used to derive a simpler analogous PDE problem in the
embedding space (for example, replacing surface intrinsic diffusion
with the bulk or \emph{Cartesian} diffusion). However, the closest
point principles hold only \textit{on the surface}. If this analogous
PDE is evolved throughout the embedding space, the restriction to the
surface may no longer be a solution of the original equation. The
approach of the explicit closest point formulation of \cite{RM2008} is
to advance the embedding PDE only by a single time step, before a
re-extension of the data is performed. At the start of the next time
step, the surface PDE will again agree on the surface with the
analogous bulk PDE. The resulting scheme can then be expressed as a
two-step explicit method in the form of \cite{RM2008, MR2008}, which
alternates between time steps of the embedding space PDE, followed by
a re-extension of the surface data. We note this approach is not a
method of lines.

An implicit version of the closest point method was introduced in
\cite{MR2009}, allowing application to stiff problems, such as those
involving biharmonic or higher-order operators. To ensure stability,
this formulation includes a stabilizing term, which can be related to
the approach of the current paper. More general forms of the closest
point function $\cp$ were introduced in \cite{MM2012}, using notions other than
Euclidean distance to determine the mapping between the surrounding
space and the surface.


\subsection{A new approach to the embedding equation}
\label{sec:embedding}
Rather than alternating between time-steps and re-extensions as in the
original closest point method \cite{RM2008}, we investigate an
alternative approach, in which a single equation can be evolved
throughout the entire embedding space, for all time, without separate
extension steps. This is achieved by creating a modified embedding
equation with a constraint.

The solution of a given surface evolution PDE is a function $u$, which
is defined only for points on the surface. We consider instead the function $v
= u \circ \cp$, which is defined for all points in a band surrounding
the surface. Based on the gradient and divergence principles
\cite{RM2008,MM2012}, we formulate a new equation for $v$. The
constraint or side condition that $v$ is a closest point extension is
enforced by adding a penalty term to the equation. We show that
evolving this new equation throughout the space to a given time $t$, and then
restricting to the surface, results in a solution of the original
surface PDE at time $t$.

If the resulting Cartesian differential operators and extension
operators are discretized in space as in \cite{MR2009}, we obtain an
ordinary differential equation in the computational band. Thus we have
a new method-of-lines formulation of the closest point method, which
can be implemented using either explicit or implicit time-stepping.

\subsection{Example - diffusion equation}
\label{subsec:diffexample}

We first illustrate the method with an example, before giving a more
detailed derivation. Let $\surf$ be a smooth closed surface embedded
in $\Real^n$, and $u$ a scalar function on $\surf$. Consider the
surface diffusion equation
\begin{equation*}
 u_t = \Delta_{\surf} u,
\end{equation*}
subject to an initial condition $u_0$.

If $B(\surf)$ is a tubular neighbourhood of the surface in $\Real^n$ (referred
to as the band), then we can define a
\textit{closest point function} $\cp: B(\surf) \rightarrow \surf$,
as in \cite{MM2012}, which maps points in the band to points on the
surface. Typically, this will be the point closest in Euclidean
distance on the surface, but this may be made more general in certain
cases \cite{MM2012}.

The \textit{extension} operator $E$ is then
defined as $Eu(x) = u \circ \cp(x)$.

The surface differential operator (Laplace--Beltrami operator) may be
replaced by a standard Laplacian using the principles in \cite{RM2008,MM2012}
\begin{equation*}
 u_t = \Delta [Eu] {\textrm{ on }} \surf.
\end{equation*}
This equation is valid only for points $x \in \surf$, since the left
hand side $u_t$ is defined only on the surface. In order to obtain an
equation defined throughout the entire band $B(\surf)$, we perform an
extension on both sides of the equation
\begin{equation*}
Eu_t = E\Delta [Eu] {\textrm{ on }} B(\surf).
\end{equation*}
We now define a function $v = Eu$ in the embedding space. Since the
operator $E$ is independent of $t$, the previous equation can be
rewritten as
\begin{equation*}
v_t = E\Delta v {\textrm{ on }} B(\surf),
\end{equation*}
subject to the condition $v = Eu$.
But if $v$ is the extension of $u$ then $v$ must also be its own
extension (see also Lemma~\ref{lemma:cpofcpiscp}) and we obtain a
system of two equations in $v$:
\begin{subequations}
\label{eq:vsystem}
 \begin{align}
 v_t &= E\Delta v,\\
 v &= Ev.
\end{align}
\end{subequations}
We will show that the solutions $v$ of this system, when restricted to the
surface $\surf$, agree with the solutions $u$ of the original
equation; i.e. $u = v|_{\surf}$.

This system could then be approximated using the two-step method of
\cite{RM2008}, where a single time step of the first equation is
carried out, and then the side condition imposed by extending the data
$v(\cdot, t_k)$ off the surface.

We propose an alternative method, in which a single equation is
solved. The side condition is added to the equation, with a constant
multiplication factor $\param$,
\begin{equation}
 v_t = E\Delta v - \param (v - Ev). \label{eq:diffusionpenalty}
\end{equation}
This equation forms the basis for the method of lines. We show that
the solutions of this single equation agree with the solutions of the
system (\ref{eq:vsystem}), for any non-zero choice of the parameter
$\param$.
However, in practice the choice of $\param$ affects the resulting
numerical methods, as investigated in Section \ref{sec:lambda}.


\section{Defining an embedding equation}
\label{sec:systemofeqs}

We will construct an equation defined in the embedding band, and show
that there is a one-to-one correspondence between solutions of this
and the original equation on the surface. In contrast to previous
formulations, the embedding equation is satisfied for all time throughout the
computational band, not only on the surface.

We first define an \textit{extension operator}, then review the
definition of the closest point gradient, divergence and Laplacian
principles. The embedding equation is then defined through the
application of these principles to the surface differential operators.

\subsection{Closest Point Principles}
In the following, let $\surf$ be a smooth surface of dimension $k$, embedded in
$\Real^n, n \geq k$, which possesses a tubular neighbourhood
or \emph{embedding band} $B(\surf) \subset \Real^n$ surrounding the
surface $\surf$. A general class of closest point functions mapping
points in the neighbourhood to the surface was introduced in \cite{MM2012}. The
closest point function based on Euclidean distance is a special
case. Given one of these closest point functions, we define the
\textit{extension operator} $E$ which acts on surface functions, and
returns a function on the embedding band $B(\surf)$.

\begin{defn}[Closest Point Extension Operator]
If $u: \surf \times \Real \rightarrow \Real$ is a scalar-valued
function on the surface, then the closest point extension $v = Eu$ is
a function $v: B(\surf) \times \Real \rightarrow \Real$ defined as 
\begin{equation*}
v(x,t) = Eu(x,t) := u(\cp(x), t), \qquad x \in B(\surf). 
\end{equation*}
This definition can then be generalized to act on functions defined on
all of $B(S)$ by operating on the restriction of the function to the
surface $\surf$
\begin{equation*}
Ev := E(v\mid_S) = v(\cp(x), t), \qquad x \in B(\surf).
\end{equation*}
Operation on a vector-valued function is defined componentwise.\label{def:cpext}
\end{defn}

The closest point principles of \cite{RM2008,MM2012} can then be
formulated using this extension operator.
\begin{principle}[Gradient Principle]
If $E$ is a closest point extension operator according to
Definition \ref{def:cpext}, then
\begin{equation*}
 \grad [ Eu ](y) = \gradsurf u(y), \quad y \in \surf,
\end{equation*}
holds for the surface gradient $\gradsurf u$ of a smooth scalar surface
function $u : \surf \rightarrow \Real$. \label{cp:grad}
\end{principle}

\begin{principle}[Divergence Principle]
If $E$ is a closest point extension operator according to
Definition \ref{def:cpext}, then
\begin{equation*}
\diverge[Eg](y) = \divsurf g(y), \quad y \in \surf,
\end{equation*}
holds for the surface divergence $\divsurf g$ of a smooth surface vector-field
$g: \surf \rightarrow \Real^n$.  \label{cp:div}
\end{principle}

\begin{principle}[Laplacian Principle]
In the case that $E$ is the particular closest point extension
operator corresponding to Euclidean distance to the surface \cite{MM2012}, then
\begin{equation*}
\lap [ Eu ] (y) = \lapsurf u(y), \quad y \in \surf,
\end{equation*}
holds for the surface Laplacian $\lapsurf u$ of $u$.  \label{cp:lap}
\end{principle}

The gradient and divergence principles above can be combined to apply to a
wider class of functions \cite{RM2008,MM2012}. In general, we assume
$A_{\surf}$ is any surface-spatial differential operator such that the
above principles can be applied to give an operator $A$ with
\begin{equation}
A(t,x,Eu) \mid_{x=y} \; = A_S(t,y,u), \qquad y \in \surf. \label{def:AS}
\end{equation}
That is, the operator $A$ (acting on functions on $B(\surf)$) is an analog of
the operator $A_{\surf}$ (acting on functions on $\surf$) where $A$
has a standard differential operator wherever $A_{\surf}$ has a
surface differential operator.

\subsection{Equivalence of surface and embedding equations}

We first give a simple lemma, which will be used frequently in the proofs below.
\begin{lemma}
The extension operator is idempotent.\label{lemma:cpofcpiscp}
\end{lemma}

{\em Proof}. Let $v$ be a closest point extension of some function
$w:B(\surf) \times \Real \rightarrow \Real$, so that $v = Ew$. Since
the closest point operator $\cp$ is a retraction,
\begin{equation*}
v(y,t) = w(\cp(y),t) = w(y,t), \qquad y \in \surf,
\end{equation*}
so $v$ and $w$ agree on the surface. Then
\begin{equation*}
E^2w = Ev = E(v\mid_S) = E(w\mid_S) = Ew. \qquad \endproof
\end{equation*}

From this, it follows that if a function $v$ can be written as the
extension of another function ($v = Ew$) then $v$ must be its own
extension ($v = Ev$).

We now show that two problems, one defined only on the surface $\surf$, and one
defined in the band $B(\surf)$, have the same solutions when restricted to the
surface.

\begin{prob}[Surface Evolution PDE]
\label{prob:surfacePDE2}
Given a smooth closed surface $\surf$ in $\Real^n$, let $u: \surf \times
[0,T) \rightarrow \Real$, be a smooth solution of the PDE
\begin{equation*}
 u_t = A_S(t,y,u) \qquad u(y,0) = u_0(y), \qquad y \in \surf, \; t \in
[0,T)
\end{equation*}
where $A_S(t,y,u)$ is a linear or nonlinear surface differential operator of
the class above.
\end{prob}

\begin{prob}[Embedding Equation]
\label{prob:embedPDE2}
Given a neighbourhood $B(\surf) \subset \Real^n$ of $\surf$, let $v:
B(\surf) \times [0,T) \rightarrow \Real$ satisfy the system of equations
\begin{subequations}
\label{eq:embedPDE2}
\begin{align}
 v_t &= EA(t,x,v) \label{eq:embedPDEa}\\
 v &= Ev, \qquad \qquad x \in B(\surf), \; t \in [0, T) \label{eq:embedPDEb}
\end{align}
\end{subequations}
with initial condition $v(x,0) = v_0(x)$. Here, $A(t,x,v)$ is a spatial differential operator on $B(\surf)$ defined from $A_S$ as in
(\ref{def:AS}).
\end{prob}

\paragraph{Remark}
Note that no additional boundary conditions have been specified for Problem
\ref{prob:embedPDE2}. By (\ref{eq:embedPDEb}), the solution everywhere
off the surface (including at the boundary of $B(\surf)$) is
determined by values on the surface. Extra boundary condition are not
necessary (although the extension (\ref{eq:embedPDEb}) is consistent
in some cases with a Neumann-type boundary condition) and imposing
artificial boundary conditions can make this problem ill-posed.

\begin{theorem}
 Suppose $\surf$ is a smooth surface embedded in $\Real^n$ and
$B(\surf)\subset \Real^n$ is a neighbourhood of the surface. Then, for
each smooth solution $u : \surf \times [0,T) \rightarrow \Real$ of the
surface PDE (Problem \ref{prob:surfacePDE2}), there exists a unique
corresponding solution $v: B(\surf) \times [0,T) \rightarrow \Real$ of
the embedding equation (Problem \ref{prob:embedPDE2}), which agrees
with $u$ when restricted to the surface $\surf$. Conversely, for every
solution $v$ of Problem \ref{prob:embedPDE2}, the restriction of $v$
to $\surf$ is a solution of Problem \ref{prob:surfacePDE2}.
\end{theorem}

\begin{proof}
To show existence, let $v(x,t) = Eu(x,t)$, where $u$ satisfies
Problem \ref{prob:surfacePDE2}. Then
\begin{equation*}
 v_t = \partial_t (Eu) = E(u_t) = E(A_S(t,y,u)),
\end{equation*}
from the surface PDE, and the fact that $E$ is time-independent. Now,
using (\ref{def:AS}) and the definition of $E$, we have that
\begin{equation*}
 E(A_S(t,y,u)) = E(A(t,x,Eu)\mid_{S}) = EA(t,x,Eu) = EA(t,x,v), \qquad
 x \in B(\surf),
\end{equation*}
so the first equation (\ref{eq:embedPDEa}) is satisfied. The second
(\ref{eq:embedPDEb}) follows from 
\begin{equation*}
Ev = E^2u = Eu = v.
\end{equation*}
Uniqueness follows from the fact that $v$ agrees with $u$ on $\surf$,
and that the off-surface values are uniquely defined by $v(x,t) =
Ev(x,t)$. The smoothness of $v$ is determined by the smoothness of the
surface and smoothness of $u$ \cite{MM2012}.

For the converse, suppose that $v$ is a solution of Problem
\ref{prob:embedPDE2}, and let $u$ be the restriction of $v$ to the
surface, $u = v|_{\surf}$. Then
\begin{equation*}
\partial_t u = \partial_t v|_{\surf} = EA(t,x,v)|_{\surf}.
\end{equation*}
Since the extension operator leaves values on the surface unchanged,
and using the closest point principles on $A$,
\begin{equation*}
\partial_t u = EA(t,x,v)|_{\surf} = A(t,x,v)|_{\surf} = A_{\surf}(t,y,
v|_{\surf}) = A_{\surf}(t, y, u),
\end{equation*}
so $u$ is a solution of Problem \ref{prob:surfacePDE2} as required.
\end{proof}


\section{From constrained embedding problem to a single equation}
\label{sec:singleeqn}

We will show that the system of embedding equations (\ref{eq:embedPDE2}) defined
in the band $B(\surf)$ has the same set of solutions as a single
equation on $B(\surf)$.

\begin{prob}
\label{prob:lambdaPDE}
Given $\param \in \Real$, let $v: B(\surf) \times [0,T) \rightarrow
\Real$ satisfy
\begin{equation}
 v_t = EA(t,x,v) - \param \left(v - Ev\right),
\qquad x \in B(\surf), \; t \in (0,T) \label{eq:lambdaPDE}
\end{equation} with initial condition $v(x,0) = v_0(x)$.
\end{prob}

\begin{theorem}
\label{thm:singleeqn}
Suppose that $v: B(\surf) \times [0,T) \rightarrow \Real$ is a solution
of the system of equations (Problem \ref{prob:embedPDE2}). Then $v$ also
satisfies the single equation (Problem \ref{prob:lambdaPDE}) for all
$\param \in \Real$.
 Conversely, if $v$ is a solution of Problem \ref{prob:lambdaPDE}
 with initial condition $v_0 = Ev_0$, then $v$ will satisfy the system
 of embedding equations (Problem \ref{prob:embedPDE2}).
\end{theorem}

\begin{proof}
The first part follows directly, since $v = Ev$ (\ref{eq:embedPDEb})
implies that the extra term multiplied by $\param$ is zero, and the
single equation (Problem \ref{prob:lambdaPDE}) becomes
equivalent to (\ref{eq:embedPDEa}).

For the converse, we operate on both sides of (\ref{eq:lambdaPDE}) with an
extension operator $E$, and use the fact that this operator is
idempotent:
\begin{equation*}
Ev_t = EA(t,x,v) - \param \left(Ev - Ev\right) = EA(t,x,v). 
\end{equation*}
Subtracting this from (\ref{eq:lambdaPDE}) gives
\begin{equation*}
(v-Ev)_t = -\param(v-Ev).
\end{equation*}
We now define a function $z = v - Ev$, to obtain an ODE for $z$:
\begin{equation*}
z_t = -\param z,
\end{equation*}
with initial condition $z_0 = v_0 - Ev_0 = 0$. This has unique
solution $z \equiv 0$. 

It follows that $v = Ev$, and so the second equation of the system
(\ref{eq:embedPDEb}) holds. Again, the extra term in
(\ref{eq:lambdaPDE}) is zero, and so the single equation is equivalent
to the system of equations. 
\qquad \end{proof}

\subsection{Remark on boundary conditions}

As above for the system of equations, the single equation
(\ref{eq:lambdaPDE}) does not require any additional boundary
conditions at the boundaries of $B(\surf)$. The imposition of other
boundary conditions could cause this problem to be ill-posed, for
example, if they are contradictory to $v = Ev$.

\subsection{The Poisson Problem}

A similar approach can be used for time-independent problems. Here, we show the
Poisson equation as an example, but this can be generalized to equations of the
form $A_S(x,t,u) = f$ for the same class of operators as above.

\begin{theorem} 
Consider the system of embedding equations obtained as above from the Poisson
equation $\lapsurf u = f$ on a surface $\surf$, 
\begin{subequations}
\begin{align}
 E\lap v &= Ef \label{eq:Poissona}\\
v &= Ev, \qquad x \in B(\surf). \label{eq:Poissonb}
\end{align}
\end{subequations}
The solutions of this system are the same as the solutions of the
single equation
\begin{equation}
E\lap v - \param(v - Ev) = Ef, \qquad x \in
B(\surf), \label{eq:Poissonembed}
\end{equation}
for any $\param \in \Real \backslash \{0\}$.
\end{theorem}

\begin{proof}
 If $v$ is a solution of (\ref{eq:Poissona}) and (\ref{eq:Poissonb}), then the
additional term in (\ref{eq:Poissonembed}) is zero, so the single
equation is satisfied. Conversely, if $v$ satisfies
(\ref{eq:Poissonembed}), then we may extend the equation to obtain
\begin{equation*}
E^2 \lap v - \param (Ev - E^2v) = E^2 f,
\end{equation*}
and using the idempotence of $E$, this is 
\begin{equation*}
E \lap v - \param (Ev - Ev ) = Ef.
\end{equation*}
It follows that
\begin{equation*}
E \lap v = Ef,
\end{equation*}
and substituting back in (\ref{eq:Poissonembed}), for any non-zero
$\param$, we have that $v = Ev$.
\qquad \end{proof}

This approach could easily be implemented for a more general differential
operator as defined in (\ref{def:AS}). Poisson problems are
investigated in another work \cite{CM}.

\section{A method-of-lines discretization}
\label{sec:MOL}

The conversion of the system of two equations --- the extended PDE
(\ref{eq:embedPDEa}) and the constraint (\ref{eq:embedPDEb}) --- into
a single equation (\ref{eq:lambdaPDE}) can now be used to define a
method-of-lines discretization. The band $B(\surf)$ is discretized
using a standard uniform Cartesian grid in $\Real^{n}$, with $N$
points in the embedding band. The vector $\vvec \in \Real^N$ is
defined as the set of values of the function $v$ at these
points. Following \cite{MR2009}, we discretize the spatial
differentiation operators on this grid using standard finite
difference schemes to obtain matrices. For example, in 2D the
Cartesian Laplacian is discretized by a matrix $\Lmat$, the standard
5-point discrete Laplacian.

Multiplication by the matrix $\Emat$ \cite{MR2009} implements the
discrete extension of a surface function, approximating the extension $E$
using interpolation on the grid points surrounding the closest point.
We note that this matrix operator is no longer idempotent, which
complicates the theory in the semi-discrete case, and indeed in this work our
analysis is mostly applied to the continuous operator.

The necessary size of the band $B(\surf)$ to contain the
differentiation and interpolation stencils is discussed in
\cite{RM2008, MR2009} and is a small multiple of the mesh parameter
$\dx$.

With the discrete operators inserted into the equation, we obtain a
system of ordinary differential equations for the vector $\mathbf{v}$
\begin{subequations}  \label{eq:semidiscrete}
\begin{equation}  \label{eq:semidiscreteheat}
 \partial_t \vvec = \Emat \Lmat \vvec  - \param(\Imat-\Emat) \vvec,
\end{equation}
or more generally
\begin{equation}
  \partial_t \vvec =
  \Emat \Amat \vvec
  - \param(\Imat-\Emat) \vvec,
\end{equation}
\end{subequations}
where $\Amat$ is the matrix discretization, in the linear case, of the
operator $A$ in (\ref{def:AS}).

This system of ODEs can then be solved using either implicit or explicit
time-stepping (or a combination). The consistency, convergence and
stability of the method will depend on the interpolation, spatial
discretization and time-stepping schemes. We discuss some of these
issues, in particular how these relate to the choice of the parameter
$\param$ in Section~\ref{sec:lambda}.

\subsection{Comparison to other semi-discrete formulations}

In the formulation of \cite{RM2008}, time steps of the discretized PDE
are alternated with an extension step:
\begin{enumerate}
\item complete one time step of $\partial_t \vvec = \Lmat\vvec$;
\item perform a re-extension $\vvec = \Emat\vvec$.
\end{enumerate}
This approach is not a method of lines; it forces the solution to be
constant in the direction normal to the surface after each time
step. In our method-of-lines approach, this requirement is imposed
with the penalty term in the PDE itself, so that no explicit
re-extension step is required.

In \cite{MR2009} an initial suggestion for a method-of-lines approach
in the particular case of the diffusion equation was the equation
$\partial_t \vvec = \Lmat \Emat \vvec$.
However, as this was seen to be unstable, a stabilized version was
proposed, solving
 $\partial_t \vvec = \Mmat \vvec$,
where the matrix $\Mmat$ was given by
\begin{equation*}
 \Mmat = \Lmat \Emat - \frac{2d}{(\Delta x)^2}(\mathbf{I}-\Emat).
\end{equation*}
At least for the diffusion equation, this is very similar to our
\eqref{eq:semidiscrete} which also solves
$\partial_t \vvec = \Mmat \vvec$
but with
\begin{equation*}
 \Mmat = \Emat \Lmat  - \param(\Imat-\Emat).
\end{equation*}
Note that the order of the matrices $\Emat$ and $\Lmat$ is reversed,
and the factor $\frac{2d}{(\Delta x)^2}$ is generalized with the
introduction of a new parameter $\param$ (although in practice we
recommend this same value for the Laplace--Beltrami operator).
In \cite{MR2009}, the discrete operator $\Mmat$ was defined based on a
special treatment of the diagonal of the discretized operator.
The new formulation \eqref{eq:semidiscrete} is based on a different
concept: we penalize the \emph{equation}, not the \emph{operator} and
this makes the approach more general.

\subsection{Nonlinear and higher-order operators}
\label{subsec:nonlinear}

Previous formulations of the stabilized operator, such as those in
\cite{MR2009} and \cite{MBR2011}, were stated for the
Laplace--Beltrami operator, and did not include a general methodology
for variable coefficient or nonlinear equations. The new method can
easily be formulated to include such operators. In
Section~\ref{sec:numerics}, we show numerical results on nonlinear
curvature-dependent diffusion and reaction-diffusion equations. As an
example of higher-order operators (which require further extensions
$E$), we consider here the biharmonic operator $\lapsurf^2$.

The surface biharmonic equation $u_t = -\lapsurf^2 u$ can be converted
using the closest point principles to the form $u_t = -\lap E
\lap E u$ on the surface. Now operating with an extension on both
sides of the equation, and substituting $v = Eu$, we have
\begin{equation*}
v_t = -E\lap E\lap v,
\qquad\text{subject to $v = Ev$.}
\end{equation*}
Forming a single equation and discretizing as before in
\eqref{eq:semidiscrete} gives the semi-discrete form
\begin{equation*}
\vvec_t = -\Emat \Lmat \Emat \Lmat \vvec - \param(\Imat - \Emat) \vvec.
\end{equation*}
We see that in general, the penalty term $\param(v - Ev)$ remains,
and an additional extension operator $E$ is included.
Note that this differs from the procedure in \cite{MR2009},
which uses the squared matrix operator $\Mmat\Mmat$ in the biharmonic case.
Both approaches seem to work in practice.
The advantages of each (or perhaps even of combinations) remains to be
studied. Fully non-linear problems also warrant further study.

\subsection{Summary of the method-of-lines approach}
The resulting algorithm can be summarized as:
\begin{enumerate}
\item Extend the surface equation into the band by applying the
  extension operator $E$, and then use Principles \ref{cp:grad} and
  \ref{cp:div} to replace surface differential operators with their
  Cartesian analogs.
\item Add the penalty term $-\param(v-Ev)$ to the PDE.
\item Use standard discretizations in space for the differential and
extension operators, and an appropriate time-stepping scheme to solve
the resulting system.
\end{enumerate}

\section{Effect of the parameter $\param$}
\label{sec:lambda}

It should be emphasized that the parameter $\param$ is not a Lagrange
multiplier; it is not necessary to solve for a value of $\param$ as
part of the solution procedure. Rather, $\param$ is a numerical
parameter that controls how strongly the constraint is imposed. The
parameter $\param$ may affect the consistency and stability of the
method. Numerical tests below suggest that a wide range of values
result in convergent schemes.

\subsection{Penalty term and zero-stability}
\label{subsec:penalty}

The term $-\param(v-Ev)$ imposes the side condition $v = Ev$ to the PDE and can
be viewed as a penalty term in the equation.
If $v$ is not constant in the direction normal to the surface, then
the term $v - Ev$ can be large, and dominate the term containing the
differential operator.
We analyze this term by considering the trivial time-dependent PDE
$u_t=0$ which leads to the equation
\begin{equation}
v_t = -\param(v-Ev). \label{eq:zerostab}
\end{equation}
Any deviation in the normal direction will be penalized by a large
right hand side.
We will show that positive values of $\param$ will return the system
to the stable equilibrium, while negative values of $\param$ may lead
to instability.

Operating on both sides of this equation with an
extension $E$, and using the idempotence of the operator gives 
$Ev_t =0$.
As in the proof of Theorem~\ref{thm:singleeqn}, we subtract this from
\eqref{eq:zerostab}, and define the function $z = v-Ev$.
We can then study this as an ODE system
\begin{equation}  \label{eq:dahlquist}
z_t = -\param z, \quad \text{ with } \; z_0 = v_0 - Ev_0,
\end{equation}
which has unique solution $z(x,t) = z_0(x) e^{-\param t}$.
If the initial condition is perturbed slightly, so
that it is no longer zero (i.e., $v_0$ is not exactly a closest point
extension), then the function $z$ will still decay to zero, provided
that $\param$ is positive.

As \eqref{eq:dahlquist} is essentially the Dahlquist test equation
\cite{HNW1993}, the region of absolute stability is determined by the
time-stepping method used. For example, for the forward Euler method,
the region of absolute stability is $|1 - \param \Delta t| \leq 1$.
With positive $\param$, this implies a time step restriction $\Delta
t \leq \frac{2}{\param}$. Likewise, for the explicit four-stage
Runge--Kutta scheme we have a stability restriction of $\Delta t \leq
\frac{2.79}{\param}$. For A-stable methods such as the implicit Euler
method, all positive values of $\param$ give stable solutions for the
trivial PDE.

Although this analysis was based on the continuous operator $E$ rather
than the discrete $\Emat$, our computations below suggest good agreement.
Thus returning to the semi-discrete problem \eqref{eq:semidiscrete},
we expect a stability restriction (when using an explicit scheme)
based on each of the two terms in the equation and we will take
$\Delta t$ based on the minimum of the two restrictions. One
reasonable strategy in choosing $\param$ is to avoid increasing the
stiffness of the system (compared to that of the Cartesian
discretization of the equivalent non-surface PDE problem).

\paragraph{\indent{Example: surface diffusion equation}}
If we discretize \eqref{eq:semidiscreteheat} using forward Euler in
time and the standard second-order scheme for the Laplacian, we might
expect a time step restriction of
\begin{equation}  \label{eq:minstab}
  \Delta t \leq \min\left(
    \frac{\Delta x^2}{2d},
    \frac{2}{\param} \right).
\end{equation}
Thus, at least for the surface diffusion equation, we can recommend a
value of $\param$ of
\begin{equation}  \label{eq:recommendlambda}
  \param = \frac{2d}{\Delta x^2},
\end{equation}
and with this choice we can expect the usual choice $\Delta t \leq
\frac{\Delta x^2}{2d}$ to result in a stable scheme
(with a factor of two to spare).

Figure~\ref{fig:zerostab} illustrates that the above theory correctly
predicts the practical stability properties. On the unit circle, we
consider the equation $u_t = \lapsurf u - u$ with semi-discrete form
$\vvec_t = \Emat \Lmat \vvec - \vvec - \param (\Imat - \Emat) \vvec$.
We discretize with forward Euler and estimate the largest possible
stable time-step; the results are very close to
\eqref{eq:minstab}. Using our suggested value of $\param$ from
\eqref{eq:recommendlambda} allows the time-step predicted by the
standard non-surface Cartesian finite difference scheme. In practice
if a larger value of $\param$ is desirable, then the time-step $\dt$
could simply be reduced for stability.

\begin{figure}
\begin{center}
\includegraphics[width = 0.48\textwidth]{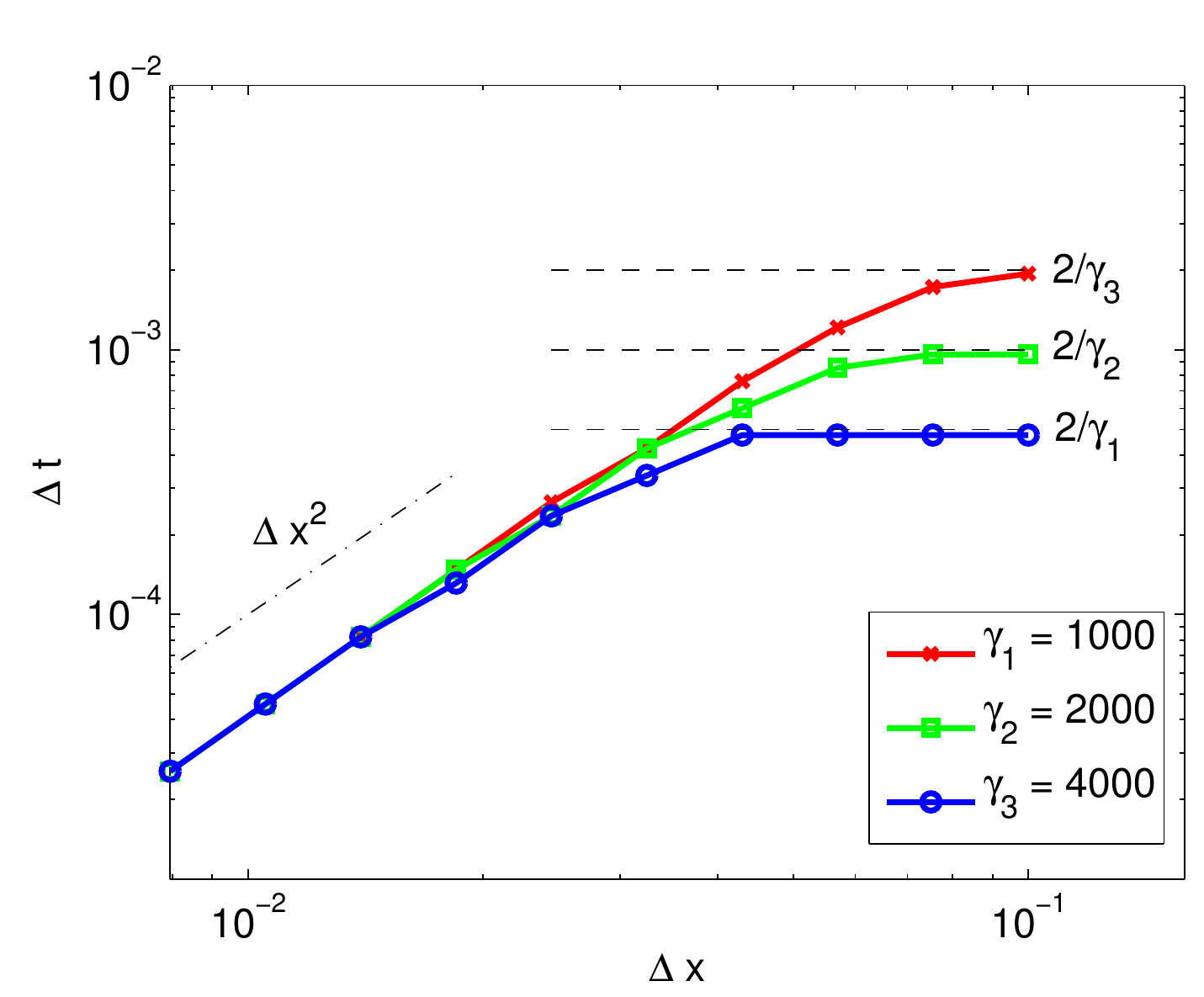}
\end{center}
\caption{Maximum value of $\dt$ giving stable solutions of $u_t =
  \lapsurf u - u$ on the unit circle, using forward Euler
  time-stepping.
}\label{fig:zerostab}
\end{figure}

\subsection{Consistency}

Requiring the method to be consistent also places certain restrictions on
the penalty parameter $\param$, as well as the interpolation order
$p$ of the extension operator. Again we consider the case of the
Laplace--Beltrami operator as an example. With a second-order spatial
discretization and first-order explicit time-stepping, the scheme can
be written
\begin{equation}
\frac{\vvec^{n+1}-\vvec^{n}}{\dt} = \Emat_p \Lmat \vvec^n - \param
(\vvec^n - \Emat_p \vvec^n).\label{eq:consistency}
\end{equation}
As before, $\Lmat$ and $\Emat_p$ are discretizations of the Laplacian
and extension operators, with polynomial interpolation of order $p$ in
the extension \cite{MR2009}. The truncation error will include standard
terms of order $O(\dt) + O(\dx^2)$ from the discretization, as well as
a term proportional to $\dx^{p+1}$ from the extension operator. If
$\param$ is chosen to scale with $\Delta x^{-\alpha}$, then the final
term in the truncation error is a contribution of
$O(\dx^{p+1-\alpha})$ from the penalty term (note that the exact
solution is an extension so the truncation error in the penalty term
is simply that of the discrete operator $\Emat_p$). Combining these
results gives an overall order of accuracy of the method of
\begin{equation*}
  O(\dt) + O(\dx^2)+ O(\dx^{p+1}) + O(\dx^{p+1-\alpha}).
\end{equation*}

For first-order consistency, it is necessary that $p \geq \alpha$, so
if $\param = O(\dx^{-2})$ (which may be required for stability), then
at least degree 2 interpolation in $p$ is needed. To maintain second
order convergence in $\dx$, at least $p=3$ is required.

Note that if in some situation, the dependence of $\param$ on $\dx$
could be freely chosen, then setting $\param$ to be a constant
and using $p=1$ should also give second order convergence.
This would be computationally more efficient, since then only
bilinear/trilinear interpolation matrices could be used. Further
details of the consistency of the closest point method are given in
\cite{MM2013}.

\subsection{Stability}

The stability of the system will also depend on the choice of
$\param$, for the equation and discretization considered. In the case that
$\param$ is zero, the side condition is not enforced at each time step.
In practice, in this case small errors in the normal direction tend to
grow over time, eventually leading to instability.
Figure \ref{fig:stability} shows how the maximum error in the solution
depends on $\param$ for the particular case of the heat equation on
the unit circle at time $t=0.5$, using forward and backward Euler
time-stepping with
$\Delta t = \frac{1}{4}\Delta x^2$ and $\Delta t = \frac{1}{4}\Delta
x$ respectively. The vertical line in the first figure is at $\param
\Delta x^2 = 8$, where the solution becomes unstable at large $\param$
due to the loss of zero-stability described in
Section~\ref{subsec:penalty}. For implicit time-stepping, the large $\param$
instability does not occur. As $\param$ becomes too small ($\param
\Delta x^2 \approx 0.1$, the solution may also become
unstable. Intuitively, this is because the penalty is not strong
enough to impose the constraint. For the heat equation, a suggested
value is $\param \approx \frac{4}{\Delta x^2}$ (and this is the same
value chosen in \cite{MR2009}). In \cite{CM}, a case is considered
where the two extension operators in the scheme (\ref{eq:consistency})
have different degrees $p$ of interpolation. For the Poisson equation
on closed curves in $\Real^2$, the scheme with two extension operators
with polynomial interpolations of order $1$ and $3$, and $\param
\approx \frac{4}{\Delta x^2}$ is shown to be second-order and stable.

\begin{figure}
\begin{center}
 \subfloat[][]{
 \includegraphics[width=0.45\textwidth]{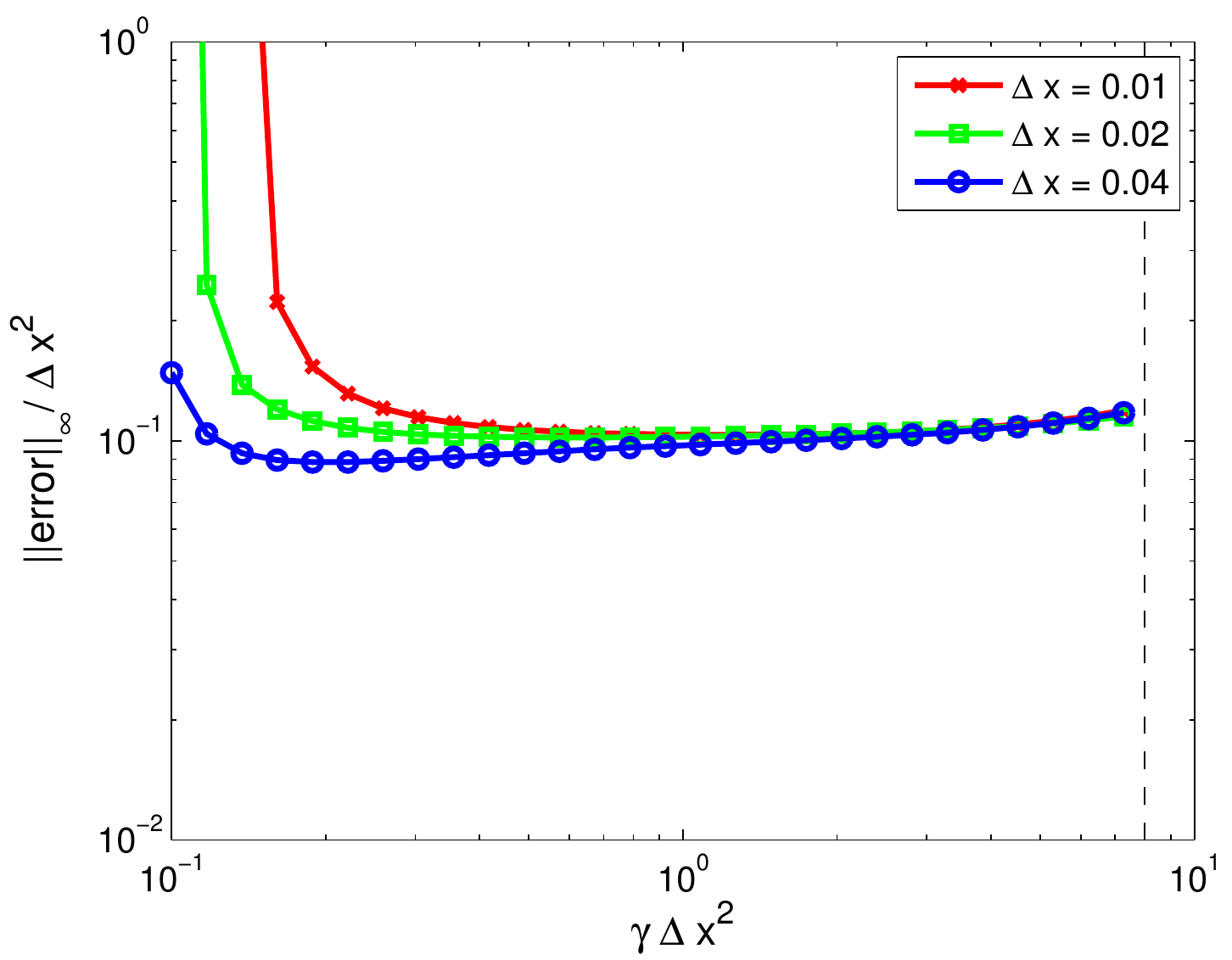}
 }\label{subfig:stabFE}
 \subfloat[][]{
 \includegraphics[width=0.45\textwidth]{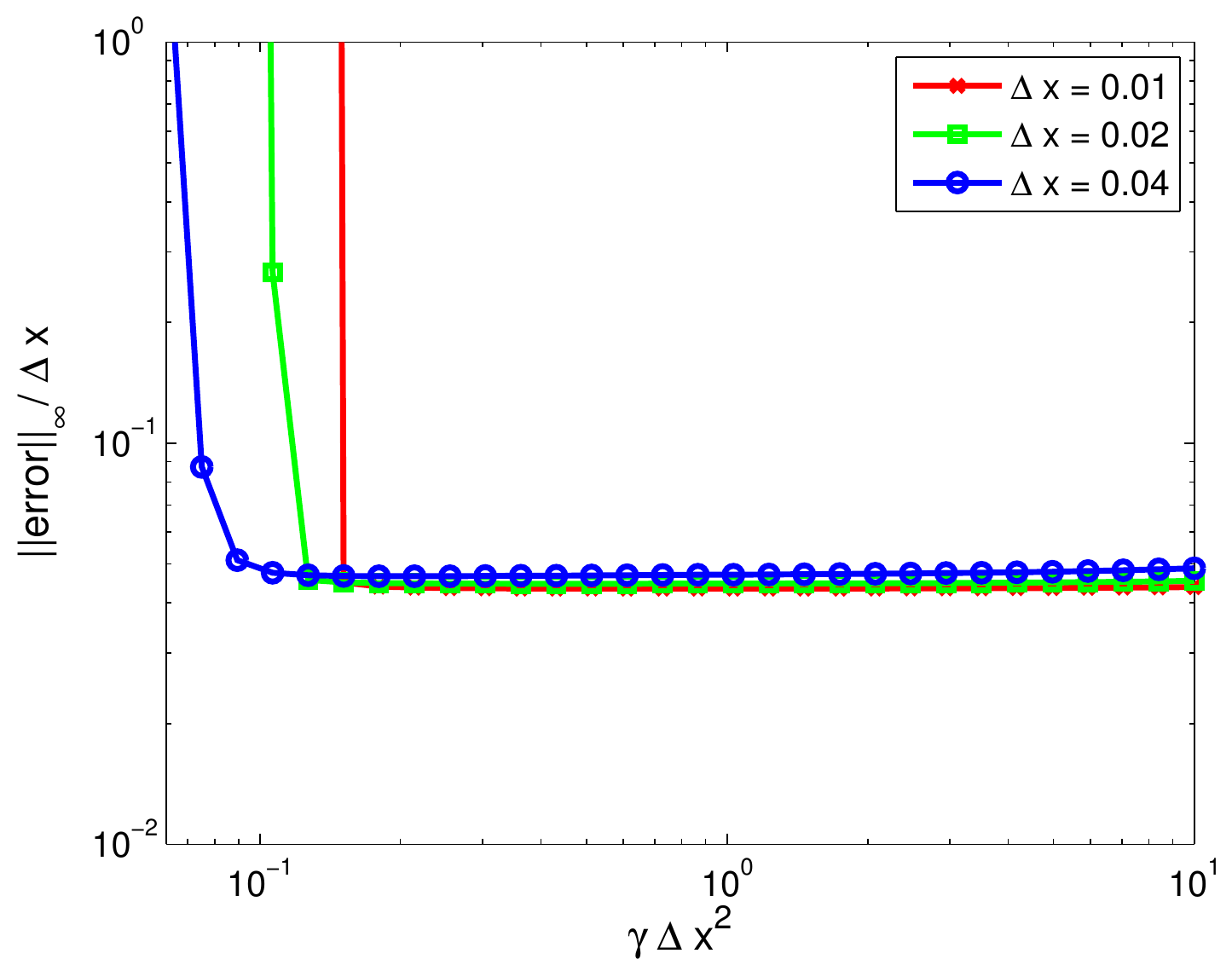}\label{subfig:stabBE}
 }
\end{center}
 \caption{Maximum error as a function of $\param \Delta x^2$, for the
 heat equation on the unit circle, using (a) explicit and (b) implicit Euler time-stepping.}
 \label{fig:stability}
\end{figure}

\subsection{Relationship to the explicit method of Ruuth \& Merriman}
Note that if we choose $\param = \frac{1}{\dt}$, the resulting
iteration for the surface heat equation will be the
same as the two-step method of \cite{RM2008}. The scheme (\ref{eq:consistency})
becomes
\begin{equation*}
\vvec^{n+1} = \Emat_p(\Delta t \Lmat \vvec^n + \vvec^n),
\end{equation*}
which corresponds to applying one step of first-order explicit
time-stepping, followed by performing an extension.

\section{Numerical examples}
\label{sec:numerics}

We demonstrate the effectiveness of the new method with various
examples in 2D and 3D.

\subsection{Diffusion equation on the unit circle and unit sphere}

The diffusion equation example of Section \ref{subsec:diffexample}
is studied on the unit circle embedded in 2D, and the unit sphere
embedded in 3D. Starting from the surface equation $u_t = \lapsurf u$,
the resulting embedding equation with the penalty term is
\begin{equation*}
 v_t = E\lap v - \param (v - Ev).
\end{equation*}
We take the standard parameterization $\sigma:[0, 2\pi) \rightarrow
\surf, \sigma(\theta) = (\cos(\theta), \sin(\theta))^T$, for the unit
circle, and write $\bar{u}(t, \theta) = u(t, \sigma(\theta))$. The
initial condition on the circle is taken to be $\bar{u}(0, \theta) =
\cos \theta + \cos 3 \theta$, giving exact solution $\bar{u}(t,\theta)
= e^{-t} \cos \theta + e^{-9t} \cos 3 \theta$. Similarly, the
parameterization of the sphere is given by $\sigma:(-\pi, \pi] \times
[-\frac{\pi}{2}, \frac{\pi}{2}] \rightarrow \surf, \sigma(\theta,
\phi) = (\cos \theta \cos \phi, \sin \theta \cos \phi, \sin
\phi)$. The initial condition is $\bar{u}(0,\theta,\phi) = \cos(\phi +
1/2)$, so that $\bar{u}(t,\theta,\phi) = e^{-2t}\cos(\phi + 1/2)$.

Standard second-order central differences are used to discretize the
Laplacian, and the order $p$ of the polynomial interpolation is varied. Figures
\ref{subfig:diffexp2D} and \ref{subfig:diffBDF3D} show convergence studies with
explicit and implicit time-stepping, using a forward Euler and BDF2
scheme respectively. The parameter $\param$ is fixed to be
$\frac{2d}{\dx^2}$, where $d$ is the dimension of the embedding space.
The solution is run in time until $t = 0.5$,
using $\dt = \frac{1}{4} \dx^2$ in the explicit case, or $\dt = \frac{1}{4}
\dx$ for the implicit BDF2 scheme. The figures demonstrate the
expected second-order convergence for $p \geq 3$. We compute the error
by restricting the solution of the embedding equation to the surface,
and computing the max-norm error over the discrete approximation to
$\surf$.

\begin{figure}
 \subfloat[][]{
 \includegraphics[width=0.45\textwidth]{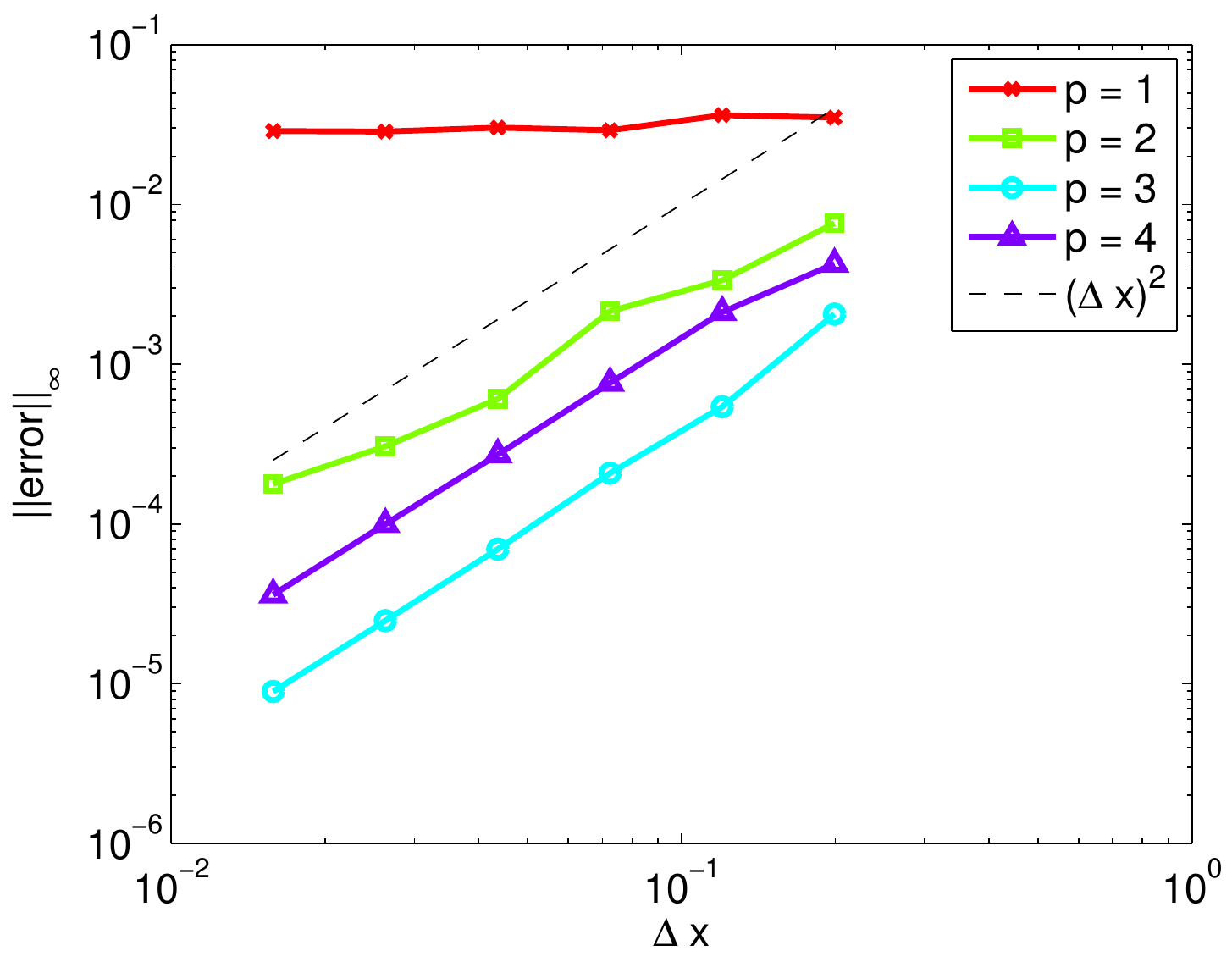}
  \label{subfig:diffexp2D}
}
 \subfloat[][]{
 \includegraphics[width=0.45\textwidth]{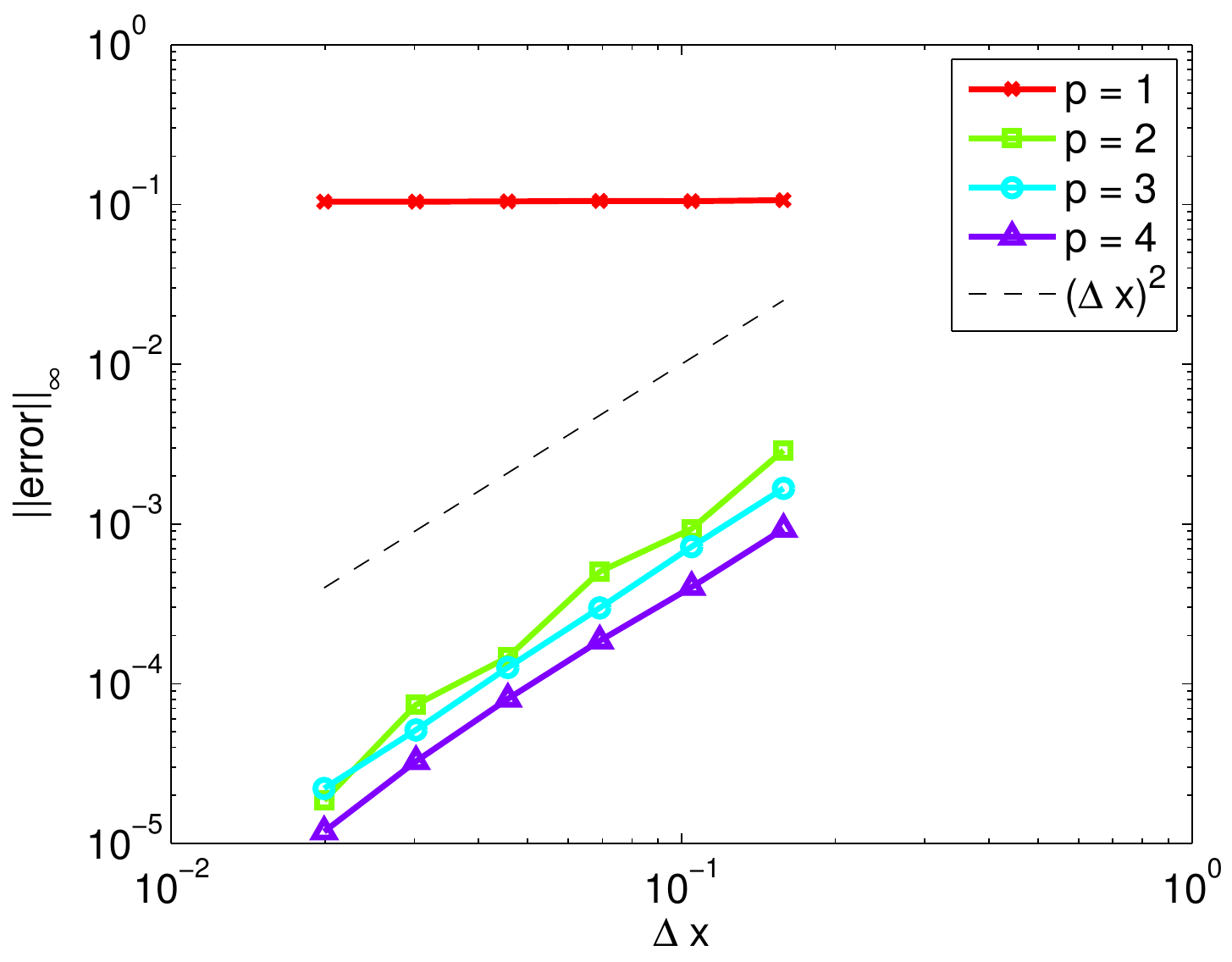}
 \label{subfig:diffBDF3D}
 }
 \caption{Numerical convergence studies for the diffusion
  equation on the unit circle (a) and unit sphere (b), using forward
  Euler and BDF2 time-stepping respectively, with $\param =
  \frac{4}{\dx^2}$.}
 \label{fig:diff}
\end{figure}

\subsection{Biharmonic equation}

As an example of a higher-order operator requiring more extensions,
consider the biharmonic equation $u_t = -\lapsurf^2 u$, again on
the unit circle in 2D. As in Section \ref{subsec:nonlinear},
the resulting embedding PDE is 
\begin{equation}
v_t = -E \lap E \lap v - \param(v - Ev)
\end{equation}
The initial condition $u(\theta,0) = \cos \theta + \cos 3
\theta$ results in the exact solution $u(\theta, t) = e^{-t} \cos
\theta + e^{-81t} \cos 3 \theta$ at time $t$.

Explicit time step restrictions become prohibitive for the higher-order
operators, so only implicit schemes are considered. We do not yet know
how to choose $\param$ in this biharmonic case; further work is
required. However, with $\param = \frac{4}{\dx^2}$ and $p \geq 4$, we
do observe second order convergence in Figure~\ref{fig:biharm}, which
shows the error in a BDF2 implicit time-stepping scheme, with $\dt =
\frac{1}{4}\dx$.

\begin{figure}
\begin{center}
\includegraphics[width=0.5\textwidth]{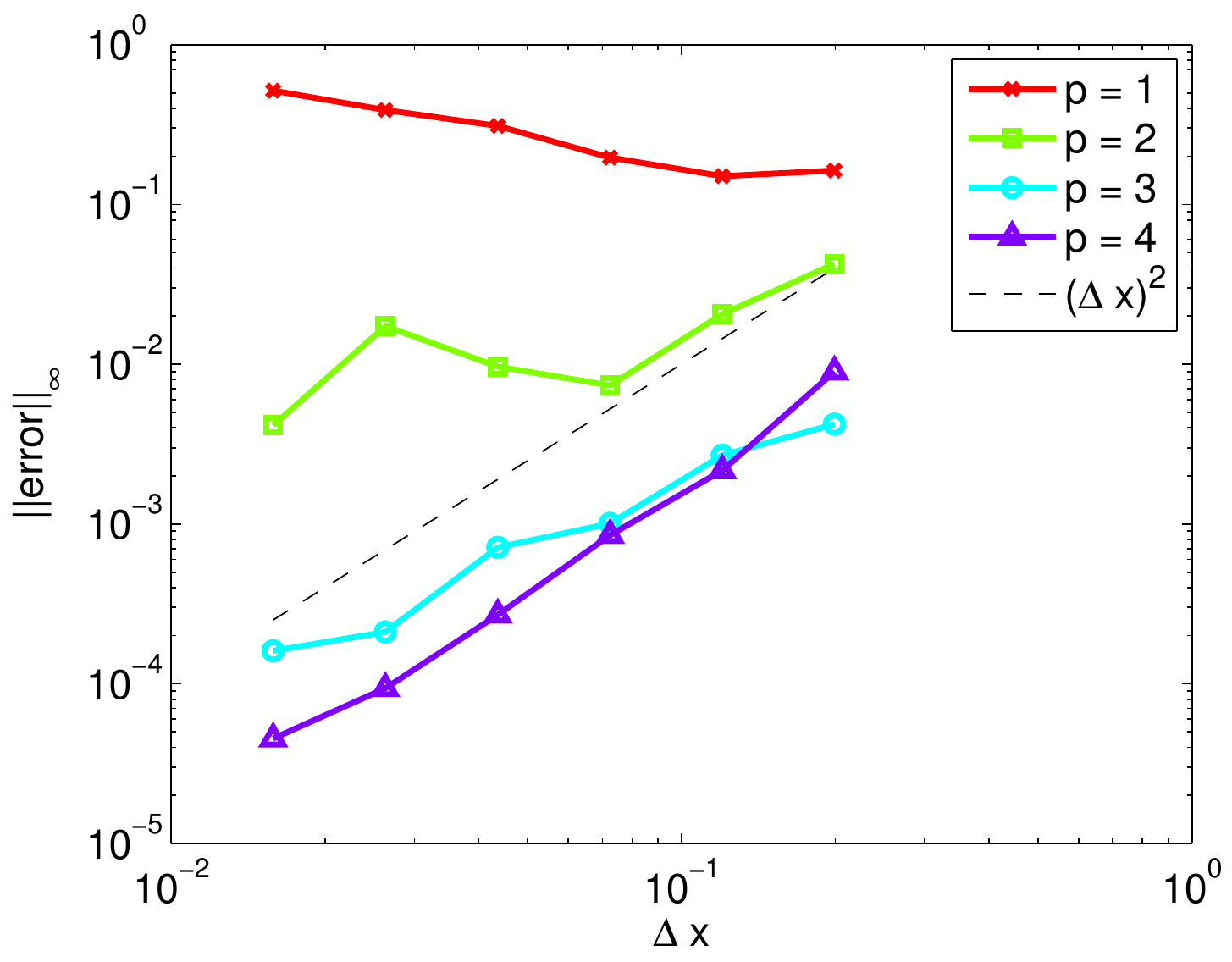}
\end{center}
\caption{Convergence study for the biharmonic
 equation on the unit circle, using BDF2 implicit time
 stepping, with $\param = \frac{4}{\dx^2}$.
}
\label{fig:biharm}
\end{figure}

\subsection{Reaction-diffusion equations on a triangulated surface}
The Gray--Scott reaction-diffusion equations are used as a model of
pattern formation \cite{GS1983, Pearson1993}. Formulated on a surface,
the equations are given by
\begin{subequations} \label{eq:surfGS}
\begin{align}
u_t &= \nu_u \lapsurf u - uv^2 + F(1-u)\\
v_t &= \nu_v \lapsurf v + uv^2 - (F+k)v
\end{align}
\end{subequations}
This example involves nonlinear terms in $u$ and $v$, which are simple
to treat with the method of lines. After extension and discretization
in space, the system of equations in the computational band becomes
\begin{subequations} \label{eq:GS}
\begin{align}
\uvec_t &= \nu_u \Emat \Lmat \uvec - \uvec\vvec^2 + F(1-\uvec)
- \param(\uvec - \Emat \uvec)\\
\vvec_t &= \nu_v \Emat \Lmat \vvec + \uvec\vvec^2 - (F+k)\vvec
- \param(\vvec - \Emat \vvec)
\end{align}
\end{subequations}
The surface is a triangulated genus 3 shape
\cite{genus3},
from which a closest point function is calculated \cite{MR2008}. This
system of equations is solved with an implicit-explicit IMEX scheme
treating the Laplace--Beltrami operators implicitly, and the nonlinear
terms explicitly \cite{Ruuth1995}. Diffusion constants used are $\nu_u
= (\Delta x)^2/9$, $\nu_v = \nu_u/2$, with parameters $k=0.063$,
$F=0.054$ \cite{munafo}. The results for $u$ at steady state are shown
in Figure \ref{fig:RD}.
\begin{figure}
\begin{center}
\includegraphics[width=0.5\textwidth]{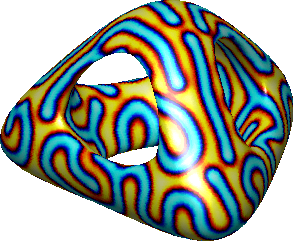}
\end{center}
\caption{Gray--Scott reaction-diffusion (\ref{eq:surfGS}) on a genus 3 surface.}
 \label{fig:RD}
\end{figure}

\subsection{Curvature-dependent diffusion on surfaces}
The geometry of the surface itself can be included in the PDE. This is
demonstrated with an example of a diffusion equation, where the
diffusivity depends on the curvature of the surface. Consider the equation
\begin{equation}
u_t(y) = \divsurf (a(y) \gradsurf u(y)),\label{eq:curvedepdiff}
\end{equation}
where we choose the inhomogeneous diffusivity $a(y)$ related to the
mean curvature $\kappa(y)$ of the surface by
\begin{equation*}
a(y) = \frac{1}{1 + |\kappa(y)|}.
\end{equation*}

We can use the closest point function representing the surface to calculate the
curvature directly. The mean curvature $\kappa$ on the surface is given by
\begin{equation}
\kappa(y) = |[\lap \cp](y)|_2.  \label{eq:meancurv}
\end{equation}

This follows from the fact that mean curvature vector $HN$ can be
written as the Laplace--Beltrami operator of the identity function on
the surface \cite{CDR2004}
\begin{equation*}
H(y) N(y) = - \lapsurf {Id_{\surf}}(y).
\end{equation*}
Here, $N(y)$ is the normal to the surface at the point $y$. Applying
the closest point principles, we have
\begin{equation*}
H(y) N(y) = -\lap (E Id_{\surf}) (y) = -[\lap \cp] (y),
\end{equation*}
since the closest point function is the extension of the identity on
the surface. The magnitude $\kappa$ of the mean curvature is found by
taking the two-norm of this expression. Following our discretization
of Section \ref{sec:MOL}, we compute the mean curvature based on
(\ref{eq:meancurv}) on the grid of the embedding space $B(\surf)$ by
\begin{align*}
  \boldsymbol\kappa &= \Emat \sqrt{
    (\Lmat \, \mathbf{cp_1})^2 + (\Lmat \, \mathbf{cp_2})^2 + (\Lmat
    \, \mathbf{cp_3})^2},
\end{align*}
where $\mathbf{cp_1}$, $\mathbf{cp_2}$, and $\mathbf{cp_3}$ are
vectors of the components of the closest point associated with each
grid point. From this we compute the vector $\mathbf{a}$ consisting of
the values of $a$ at each grid point.

The PDE (\ref{eq:curvedepdiff}) is simple to solve numerically using
the method of lines approach detailed in Sections~\ref{sec:singleeqn}
and~\ref{sec:MOL}.  The embedded surface with penalty term is
\begin{equation*}
v_t = E \, \diverge \, (a \grad v) - \param(v - Ev).
\end{equation*}
Now a standard scheme is used to discretize the variable coefficient
diffusion term which yields the semi-discrete form
\begin{equation*}
  \vvec_t = \Emat \Big[
                   \Dmat^x_b \big(\Amat^x_{\!f} \mathbf{a} \;
                   \Dmat^x_{\!f} \vvec\big) +
                   \Dmat^y_b \big(\Amat^y_{\!f} \mathbf{a} \;
                   \Dmat^y_{\!f} \vvec\big) +
                   \Dmat^z_b \big(\Amat^z_{\!f} \mathbf{a} \;
                   \Dmat^z_{\!f} \vvec\big) \Big]
    - \param\big(\Imat - \Emat\big)\vvec,
\end{equation*}
where $\Dmat_b$ and $\Dmat_{\!f}$ are the backward and forward
finite difference matrices in the direction indicated by the
superscripts.
Similarly, the $\Amat_{\!f}$ matrices refer to forward two-point
averages of the point-wise diffusivity values.
That is, the half-point diffusivities are approximated by the
averages:
\begin{equation*}
  a_{i+\hf,j,k} \approx \frac{a_{i+1,j,k} + a_{ijk}}{2}, \quad
  a_{i,j+\hf,k} \approx \frac{a_{i,j+1,k} + a_{ijk}}{2}, \quad
  a_{i,j,k+\hf} \approx \frac{a_{i,j,k+1} + a_{ijk}}{2},
\end{equation*}
This scheme can then be evolved with explicit Euler time-stepping.

Figures \ref{subfig:ellipse} and \ref{subfig:snowflake} show the
curves used to demonstrate this approach: an ellipse and the curve
parameterized by $x = (1 + \frac{1}{3} \cos (6 s)) \cos s , y = (1 +
\frac{1}{3} \cos (6 s)) \sin s$. Initial conditions are $u(s,0) =
\cos(3s)$. The resulting solutions at time $t = 0.5$ are shown in
Figures \ref{subfig:diffusionellipse} and
\ref{subfig:diffusionsnowflake}. Solutions are also computed using
Chebfun \cite{chebfunv4} (based on the parameterization) and plotted
for comparison.

\begin{figure}
\begin{center}
\subfloat[][]{
\includegraphics[height=0.21\textheight]{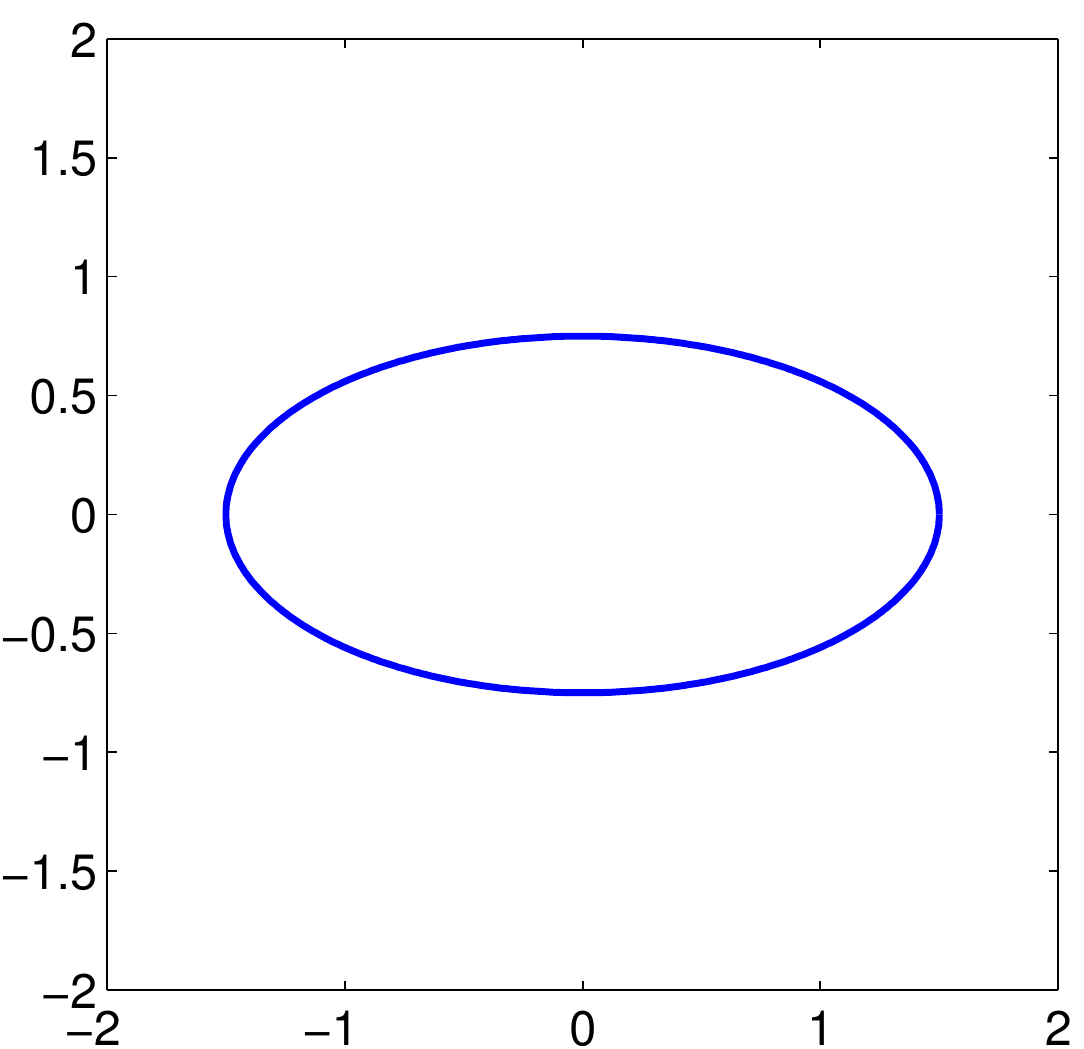}
\label{subfig:ellipse}
}
\subfloat[][]{
\includegraphics[height=0.21\textheight]{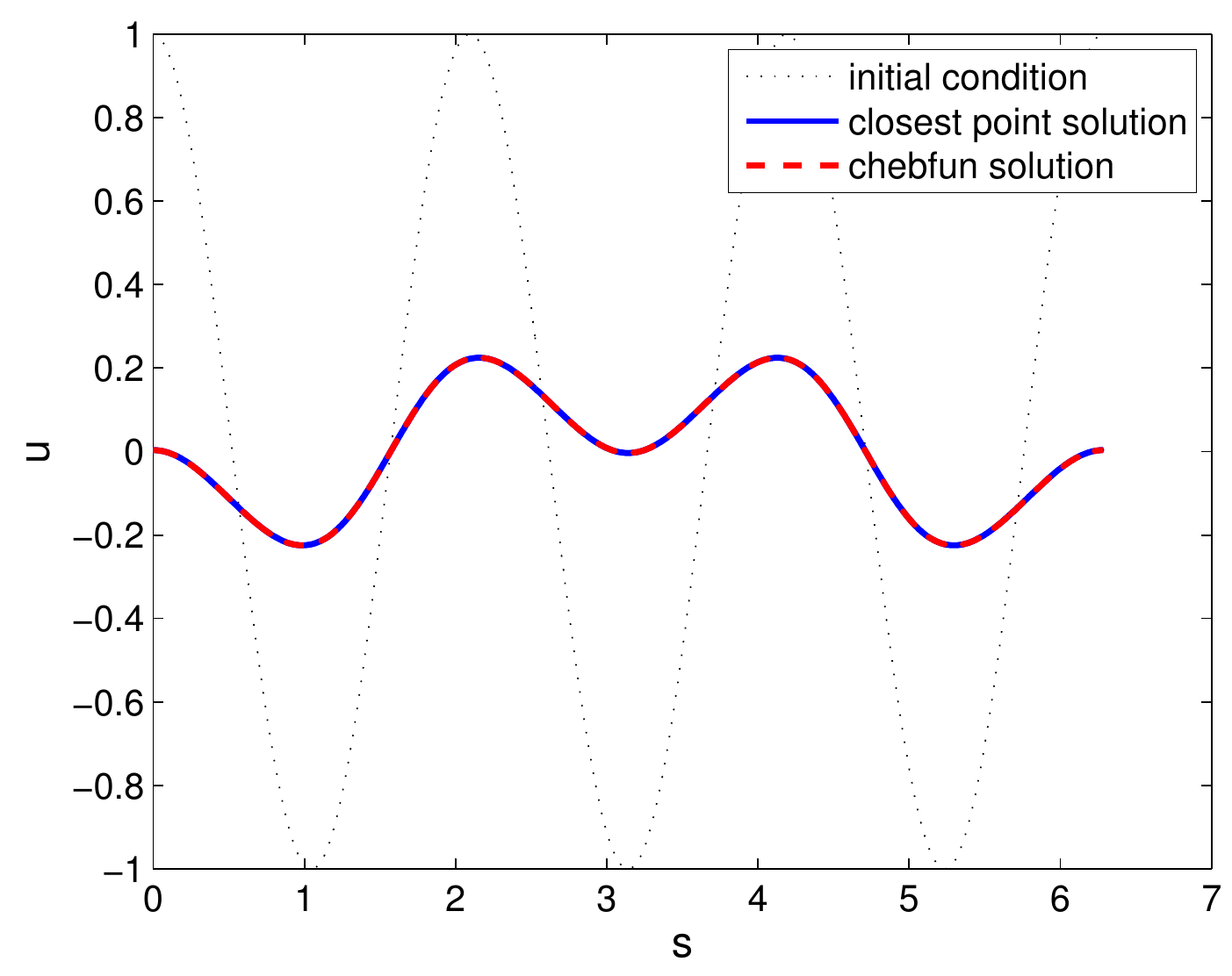}
\label{subfig:diffusionellipse}
}\\
\subfloat[][]{
\includegraphics[height=0.21\textheight]{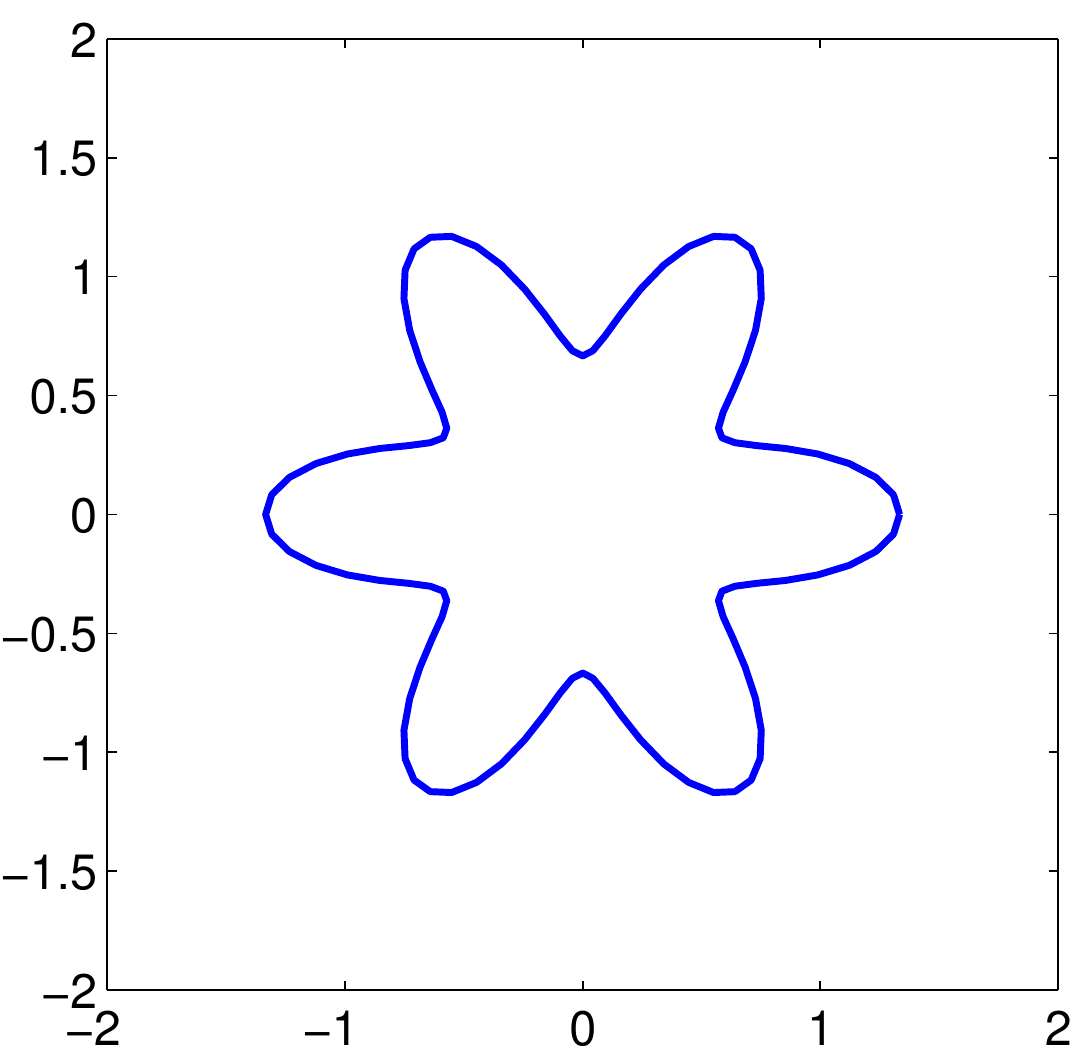}
\label{subfig:snowflake}
}
\subfloat[][]{
\includegraphics[height=0.21\textheight]{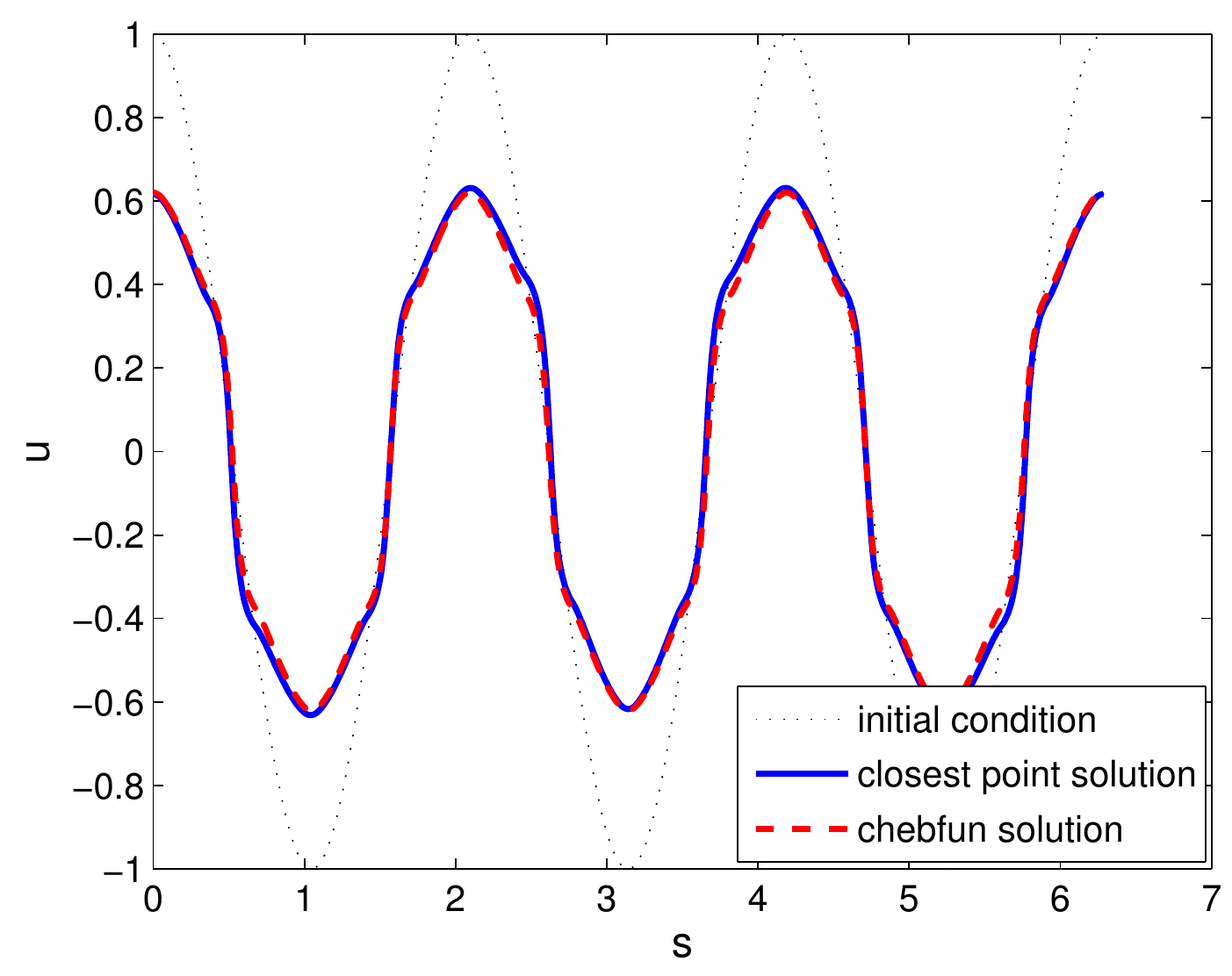}
\label{subfig:diffusionsnowflake}
}
\end{center}
\caption{Curvature-dependent diffusion on surfaces of non-constant
  curvature. Curves embedded in $\Real^2$ (a) and (c). Results of
  curvature-dependent diffusion (\ref{eq:curvedepdiff}) at time
  $t=0.5$, using our method and Chebfun \cite{chebfunv4} (b) and (d).}
 \label{fig:curvsnowflake}
\end{figure}

\subsection{Reaction-diffusion with curvature-dependent parameters}
\begin{figure}
\begin{center}
\subfloat[][]{
\includegraphics[width=0.5\textwidth]{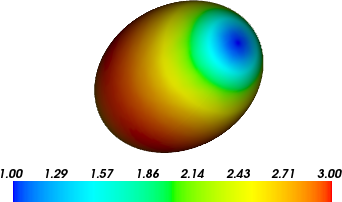}
\label{subfig:RDellipsecurv}
}\\
\subfloat[][]{
\includegraphics[width=0.3\textwidth]{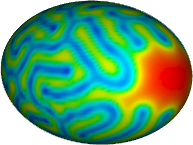}
\label{subfig:RDellipse1}
}
\hspace{0.2\textwidth}
\subfloat[][]{
\includegraphics[width=0.27\textwidth]{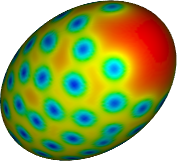}
\label{subfig:RDellipse2}
}
\end{center}
\caption{Curvature-dependent reaction-diffusion on an ellipsoid.
  Ratio of diffusion coefficients (a) --- inversely proportional to surface
  curvature, stripe (b) and spot (c) formation in
  regions of low curvature using the Gray--Scott model
  (\ref{eq:GS}).} \label{fig:RDcurvellipsoid}
\end{figure}
\begin{figure}
\begin{center}
\subfloat[][]{
\includegraphics[width=0.42\textwidth]{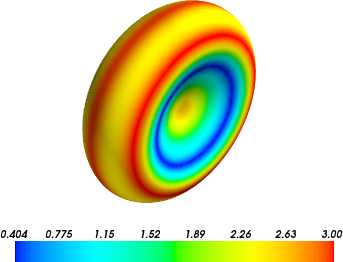}
\label{subfig:RBFcurv}
}
\subfloat[][]{
\includegraphics[width=0.29\textwidth]{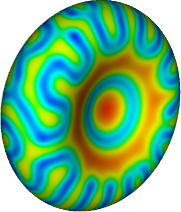}
\label{subfig:RBF1}
}
\subfloat[][]{
\includegraphics[width=0.25\textwidth]{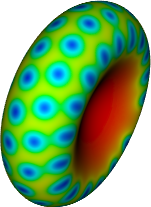}
\label{subfig:RBF2}
}\\
\end{center}
\caption{Curvature-dependent reaction-diffusion on a red blood cell
  shape.
  Ratio of diffusion coefficients (a) --- proportional to surface
  curvature, stripe (b) and spot (c) formation in
  regions of high curvature using the Gray--Scott model
  (\ref{eq:GS}).} \label{fig:RDcurvRBC}
\end{figure}
Through dependence on curvature, the geometry of the surface could
influence systems such as reaction-diffusion equations. In diffusion-driven
instability, the difference in diffusion coefficients of two chemical
species drives an instability leading to pattern formation
\cite{Pearson1993}. If the diffusivities vary across the surface,
patterns may form only in certain areas.

In the Gray--Scott model above, the ratio of diffusion coefficients
$\nu_v = \frac{\nu_u}{2}$ is used to form a patterned steady
state. With equal coefficient values, no patterns are formed. We now
consider a case where $\nu_v$ varies with curvature of the surface.

The approaches of the two previous numerical examples are
combined. The Gray--Scott scheme (\ref{eq:GS}) is solved on a surface of
non-constant curvature, with $\nu_v$ related to $\nu_u$ by
\begin{equation*}
\nu_v = \nu_u / \big(3 - \tfrac{2}{c_1 - c_2}(\kappa - c_2)\big),
\end{equation*}
where $c_1$ and $c_2$ are the maximum and minimum curvatures of the
surface. At areas of low curvature, the ratio will be close to 3,
while areas of high curvature will have equal coefficients. 

We expect patterns to form preferentially in low-curvature areas, as
demonstrated in Figure \ref{fig:RDcurvellipsoid}. Initial conditions
are taken to be the steady state $(u_0,v_0) = (1,0)$, with random
Gaussian noise added. Figure \ref{fig:RDcurvellipsoid}a shows the
ratio of the diffusivities calculated from the mean curvature of an
ellipsoid. Steady states for $u$ demonstrating spot and stripe formation
on this surface are shown in Figures \ref{fig:RDcurvellipsoid}b and
\ref{fig:RDcurvellipsoid}c. Parameters used are $F = 0.026$,
$k = 0.061$ for spots and $F = 0.054$, $k = 0.063$ for stripes, as
expected on flat domains \cite{munafo}. A similar system is solved on
a parameterized red blood cell shape (derived in \cite{EF1972} and
used with reaction-diffusion models in \cite{FW2012}), this time with
nonequal coefficients at areas of high curvature. Figure
\ref{fig:RDcurvRBC} shows spot and stripe formation on the
high-curvature regions of the surface.

\section{Conclusions}
\label{sec:conclusions}

We have introduced a new formulation of an embedding method for
solving partial differential equations (PDEs) on surfaces, based on
the closest point representation. Our formulation results from the
addition of a penalty term to the surface PDE, which helps ensure that
the solution in the embedded space stays constant in the normal direction.
Like the original closest point method of Ruuth and Merriman, the
method is simple and very general with respect to surface geometry,
dimension and co-dimension.

Compared to previous attempts to construct an implicit closest point
method, our method has an advantage in that it works for variable
coefficient and nonlinear PDEs. Because the method allows a
method-of-lines discretization, it can be used with either implicit or
explicit timestepping (and, although not our focus here, for elliptic problems).
Our approach also seems simpler to analyze.

The solutions of the new embedding equation, when restricted to the
surface, are shown to correspond with a one-to-one map to the
solutions of the original PDE. The modified equation involves a
parameter; we show that, while in the continuous problem any value
will work, in numerical discretizations the value is important.
In particular, the effect of this penalty parameter on stability is
analyzed, and numerical studies of convergence are shown for the
Laplace--Beltrami operator and surface biharmonic operators.
Examples demonstrate the effectiveness of the method for nonlinear
operators on various parameterized and triangulated surfaces, in
particular relating to curvature-dependent diffusion.

Future work could investigate fully nonlinear problems and the role of
the penalty parameter in higher-order problems, for example, a more
thorough treatment of surface biharmonic problems.

\subsection*{Acknowledgements}
We thank Dr Anotida Madzvamuse (Sussex) for suggesting the application
of curvature-influenced reaction-diffusion.

\bibliography{cpmol}{}
\bibliographystyle{siam}

\end{document}